\documentclass[10pt]{amsart}
\usepackage{a4}
\usepackage{amssymb}
\usepackage{array}
\usepackage{booktabs}
\usepackage{multirow}
\usepackage{hhline}

\newtheorem*{thma}{Theorem A}
\newtheorem*{thmb}{Theorem B}
\newtheorem{thm}{Theorem}

\newtheorem{lem}[thm]{Lemma}
\newtheorem{prop}[thm]{Proposition}
\newtheorem{defn}[thm]{Definition}
\newtheorem{ex}{Example}

\newtheorem{rem}[thm]{Remark}

\numberwithin{thm}{section}
\numberwithin{equation}{section}

\newcommand{\la}{\langle}
\newcommand{\ra}{\rangle}

\newcommand{\A}{\mathcal{A}}

\newcommand{\B}{\mathcal{B}}

\newcommand{\Comp}{\mathbb{C}}

\newcommand{\F}{\mathcal{F}}

\newcommand{\z}{\mathbb{Z}}

\newcommand{\om}{\omega}

\newcommand{\prt}{\widehat{\otimes}}

\newcommand{\injt}{\otimes_{\min}}

\begin{document}
\title{New deformations of convolution algebras and Fourier algebras on locally compact groups}

\author{Hun Hee Lee}
\author{Sang-gyun Youn}

\address{Hun Hee Lee : Department of Mathematical Sciences, Seoul National University,
San56-1 Shinrim-dong Kwanak-gu, Seoul 151-747, Republic of Korea}
\email{hunheelee@snu.ac.kr}

\address{
Sang-gyun Youn: Department of Mathematical Sciences, Seoul National University,
San56-1 Shinrim-dong Kwanak-gu, Seoul 151-747, Republic of Korea}
\email{yun87654@snu.ac.kr}

\keywords{Fourier algebra, Convolution algebra, operator algebra, Beurling algebra}
\thanks{2000 \it{Mathematics Subject Classification}.
\rm{Primary 43A20, 43A30; Secondary  47L30, 47L25}\\
\rm{The first named author is supported by the Basic Science Research Program through the National Research Foundation of Korea (NRF), grant NRF-2015R1A2A2A01006882.
The second named author is supported by TJ Park Science Fellowship.}}

\begin{abstract}
In this paper we introduce a new way of deforming convolution algebras and Fourier algebras on locally compact groups. We demonstrate that this new deformation allows us to reveal some informations of the underlying groups by examinining Banach algebra properties of deformed algebras. More precisely, we focus on representability as an operator algebra of deformed convolution algebras on compact connected Lie groups with connection to the real dimension of the underlying group. Similary, we investigate complete representability as an operator algebra of deformed Fourier algebras on some finitely generated discrete groups with connection to the growth rate of the group. 
\end{abstract}

\maketitle

\section{Introduction}
For a locally compact group $G$ it has been a long tradition to investigate its associated Banach algebras, namely the convolution algebra $L^1(G)$ and the Fourier algebra $A(G)$, in the hope that we could find connections between Banach algebraic properties of $L^1(G)$ (or $A(G)$) and the group properties of $G$. This line of research is based on the fundamental result of Wendel (\cite{Wen}) (resp. Walter (\cite{Walter})) saying that for two locally compact groups $G$ and $H$, the algebras $L^1(G)$ and $L^1(H)$ (resp. $A(G)$ and $A(H)$) are isometrically isomorphic if and only if $G$ and $H$ are isomorphic as topological groups. Of course, making concrete connections between two objects is a completely different task, and here is one of the most succesful examples of such connections. The celebrated results by B.E. Johnson (\cite{Joh}) and Z.-J. Ruan (\cite{Ruan}) tell us that $L^1(G)$ is amenable as a Banach algebra if and only if $G$ is amenable if and only if $A(G)$ is operator amenable as a completely contractive Banach algebra. Recall that $G$ is called amenable if $L^\infty(G)$ has a left invariant mean and a (completely contractive) Banach algebra $\A$ is called (operator) amenable if every (completely) bounded derivation $D:\A \to X^*$ for any (operator) $\A$-bimodule $X$ is inner (i.e. there is $\phi\in X^*$ such that $D(a) = \phi \cdot a - a \cdot \phi$, $a\in \A$).

The list of such connections continues, but at the same time there certainly are limitations. One of the attempts we could try at this moment is to consider modified versions of $L^1(G)$ and $A(G)$ expecting further connections between group properties and Banach algebraic properties. The construction of weighted convolution algebras begins with a choice of Borel measurable (or continuous) weight function $w : G \to (0,\infty)$ which is sub-multiplicative (i.e. $w(xy) \le w(x)w(y)$, $x,y\in G$ a.e.) Now we consider the weighted $L^1$ space
	$$L^1(G,w) := \{f: G \to \Comp\; |\; ||f||_{L^1(G,w)} :=\int_G |f(x)|w(x) dx <\infty \}.$$
The sub-multiplicativity of $w$ ensures that $L^1(G,w)$ is still a Banach algebra with respect to the convolution product. In other words, weighted convolution algebras are obtained by modifying the norm structure via {\it multiplying the weight function} when we calculate the $L^1$-norm but essentially keeping the same algebra multiplication, which is convolution in this case. As is expected there are results establishing connections between weighted algebras and the groups. Recall that a Banach algebra $\A$ is called {\it representable as an operator algebra} if there is an operator algebra $\B$ (i.e. a closed subalgebra of $B(H)$ for some Hilbert space $H$) and a bijective bounded isomorphism $T : \A \to \B$ with bounded inverse $T^{-1}$. We define {\it complete representability as an operator algebra} of a completely contractive Banach algebra similarly. In \cite{LSS} it is proved that $\ell^1(G, \om_\beta)$ is representable as an operator algebra if $\beta > \frac{k_0+1}{2}$, where $G$ is a finitely generated discrete group with polynomial growth of order $k_0$. Note that weighted convolution algebras, in general, have been studied extensively. See \cite{Gr, DL} and the references therein, for example.

The corresponding story for Fourier algebras has begun quite recently by Ludwig/Spronk/Turowska (\cite{LST}) and Lee/Samei (\cite{LS}). Weighted Fourier algebras follow the same philosophy of modification with more involved technicalities, and there are results connecting Banach algebraic properties and group properties. In \cite{GLSS} it is proved that for a compact connected Lie group $G$ the weighted Fourier algebra $A(G, w)$ is completely representable as an operator algebra if $w$ is a ``polynomially growing weight'' whose growth of order is strictly greater than $\frac{d(G)}{2}$, where $d(G)$ is the real dimension of the Lie group $G$. These results show us that group informations of $G$ such as polynomial growth rate or real Lie group dimension are reflected in weighted convolution (Fourier) algebras.

The current paper deals with a different way of modifying $L^1(G)$ and $A(G)$. The main difference from weighted versions is that we would like to {\it multiply a certain fixed ``function'' to the Fourier transform} of the given function. Suppose for the moment that $G$ is abelian and consider a Borel measurable $w: \widehat{G} \to (0,\infty)$, where $\widehat{G}$ is the dual group of $G$. Then, we define the deformed $L^1$-norm by
	$$||| f ||| := ||\F^{-1}(w\cdot \F(f))||_{L^1(G)},$$	
for ``nice'' functions $f \in L^1(G)$, where $\F$ is the group Fourier transform on $G$. If we set $g= \F^{-1}(w\cdot \F(f))$, then $f = \F^{-1}(\frac{1}{w}\cdot \F(g))$. In other words we are looking at the map
	$$\Phi: L^1(G) \to L^1(G),\;  g \mapsto \F^{-1}(\frac{1}{w}\cdot \F(g))$$
and $|||f||| = ||g||_{L^1(G)}$ with $f = \Phi(g)$. Of course, we want $\Phi$ to be well-defined and bounded, which means that $\Phi$ is nothing but a (Fourier) multiplier on $L^1(G)$. From a classical result on multipliers on $L^1(G)$ we know that there must be a complex measure $\mu\in M(G) \cong ML^1(G)$ such that
	$$\Phi(g) = \mu * g,\; g\in L^1(G).$$ 
This informal observation suggests us the deformed $L^1$-space $L^1_\mu(G)$ given by
	$$L^1_\mu(G) = \mu*L^1(G)$$
with the norm $||\mu*g||_\mu = ||g||_{L^1(G)},\; g\in L^1(G)$. This definition can be easily extended to the case of general locally compact groups.  The case of deformed Fourier algebra follows the same idea, so that we begin with an element in the multiplier algebra $M_{cb}A(G)$. See the detailed rigorous definitions in Section \ref{sec-construction}.

Given the above new deformations we would like to focus again on (complete) representability as an operator algebra expecting that we could extract similar informations on the underlying groups. Indeed, we prove the following results in this paper. Let $G$ is a compact connected Lie group and $\nu_\alpha$ is the probability measure whose Fourier coefficients are polynomially decreasing of order $\alpha$ on $\widehat{G}$ (see Section \ref{subsec-compact-connected} for the precise definitions).
\begin{thma}
Let $G$ is a compact connected Lie group. The algebra $L^1_{\nu_\alpha}(G)$ is (completely) representable as an operator algebra if and only if $\alpha > \frac{d(G)}{2}$, where $d(G)$ is the real dimension of $G$.
\end{thma}
We also have a corresponding result for the dual setting. Let $G$ be a finitely generated discrete group belonging to a certain class of groups (more precisely, $G$ is either $\z^n$, a Coxeter group, or a hyperbolic group). Let $W_\alpha$ and $w_t$ be the functions in $M_{cb}A(G)$ decreasing polynomially of order $\alpha$ and exponentially, respectively, with respect to the canonical word length function. See Section \ref{subsec-f.g.} for the precise definition. Then we have the following result.
\begin{thmb}
Let $G$ be a discrete group described in the above.
	\begin{enumerate}
		\item
		Suppose $G$ is of polynomial growth of order $k_0$. The algebra $A_{W_\alpha}(G)$ is completely representable as an operator algebra if and only if $\alpha > \frac{k_0}{2}$.
		\item
		Suppose that $G$ is exponential growing with the growth rate $\lambda$. The algebra $A_{w_t}(G)$ is completely representable as an operator algebra if  $t> \frac{\log \lambda}{2}$ and $A_{w_t}(G)$ is not completely representable as an operator algebra if  $t< \frac{\log \lambda}{2}$.
	\end{enumerate}	
\end{thmb}
These results also show us that the polynomial (or exponential) growth rate of some finitely generated groups and the real dimension of compact connected Lie groups can be {\it precisely detected} by examining complete representability as an operator algebra of the corresponding deformed algebras.

There are a few advantages of new deformations compared to the weighted versions. First, some informations on the groups can be {\it precisely detected} as the above theorems, whilst we only have partial results in the theory of weighted algebras (\cite{GLSS, LSS}). Secondly, new deformations can be applied for arbitrary locally compact groups in contrast to fact that the weighted versions had some limitations on the choice of groups. For example, when $G$ is compact the weighted convolution algebra $L^1(G,w)$ is isomorphic to $L^1(G)$ as Banach algebras since every weight function $w$ is known to be (\cite[Lemma 1.3.3]{Ka}) bounded and bounded away from zero (in other words, equivalent to the constant 1 function). For the same reason weighted Fourier algebras on discrete groups have never been investigated. However, we do not have this restriction for $L^1_\mu(G)$ and $A_w(G)$. In this paper we usually focus on the compact $G$ for $L^1_\mu(G)$ and discrete $G$ for $A_w(G)$, but the general case will be discussed in a subsequent paper.

This paper is organized as follows. In Section 2 we collect some preliminaries we need. In Section 3 we define deformed algebras $L^1_\mu(G)$ and $A_w(G)$ for a locally compact group $G$. In Section 4 we focus on $L^1_\mu(G)$ for a compact group $G$ and prove that representability as an operator algebra is closely related to the square-integrability of the deformation measure $\mu$. Moreover, we apply this to establish connections between representability as an operator algebra of $L^1_\mu(G)$ and the dimension of $G$ when $G$ is a compact connected Lie group. In Section 5 we turn our attetion to the case of $A_w(G)$ for discrete groups. We also prove a general result saying that representability as an operator algebra is equivalent to the square-summability of the deformation function $w$. We apply this to a certain class of finitely generated groups and show that complete representability as an operator algebra of $A_w(G)$ is closely related to the growth rate of $G$.

\section{Preliminaries}

\subsection{Operator spaces}

We will assume that the reader is familiar with standard operator space theory including injective, projective and Haggerup tensor products of operator spaces, whcih we denote by $\injt$, $\prt$ and $\otimes_h$, respectively. We will also frequently use a dual version of Haagerup tensor product, namely the extended Haggerup tensor product. The extended Haggerup tensor product of dual operator spaces $E^*$ and $F^*$ will be denoted by
	$$E^*\otimes_{eh} F^*$$
which is given by $(E\otimes_h F)^*$ in \cite{BS}. There are several characterizations of $\otimes_{eh}$, but we will only be using the following two aspects. First, for $X\in M_n(E^*\otimes_{eh} F^*)$ we have,
	$$ ||X||_{eh}=\min\left \{\left \|A\right \| \left \|B \right \| \right \},$$ where the minimums runs over all possibile factorization satisfying $\displaystyle X= A\odot B$ with $A\in M_{n,I}(E^*)$ and $B\in M_{I,n}(F^*)$ and $\odot$ is the Haagerup product given by $(A_1\otimes A_2)\odot (B_1\otimes B_2) = A_1B_1\otimes A_2\otimes B_2$ for $A_1\in M_{n,I},B_1\in M_{I,n}$ and $A_2\in E^*$, $B_2\in F^*$. Note that the index set $I$ could be arbitrary. See \cite[Theorem 2.4]{Spr}. Secondly, if $E^* \subseteq B(H)$ for some Hilbert space $H$, then we have a completely isometric embedding
	$$E^*\otimes_{eh} F^* \hookrightarrow CB^\sigma(B(H), B(H)),\; A\otimes B \mapsto T_{A,B},$$
where $CB^\sigma(B(H), B(H))$ refers to the space of all $w^*$-$w^*$-continous completely bounded maps and $T_{A,B}(X) = AXB$, $X\in B(H)$.
	
A Banach algebra $\A$ with the algebra multiplication map
	$$m : \A \otimes_\gamma \A \to \A,$$
where $\otimes_\gamma$ is the projective tensor product of Banach spaces, is called a {\it completely contractive Banach algebra} if $\A$ is endowed with an operator space structure and the map $m$ extends to a complete contraction
	$$m : \A \prt \A \to \A.$$
Note that any operator algebra carries a natural operator space structure, which makes it a completely contractive Banach algebra. Operator algebras form a quite distinctive class of completely contractive Banach algebras. In the category of completely contractive Banach algebras we have the following characterization of operator algebras by Blecher (\cite{B95}).

	\begin{thm}
	Let $\A$ be a completely contractive Banach algebra with the algebra multiplication
		$$m : \A \prt \A \to \A.$$
Then, $\A$ is completely representable as an operator algebra if and only if the multiplication map extends to a completely bounded map
		$$m : \A \otimes_h \A \to \A.$$
	\end{thm}

\subsection{Relevant spaces in abstract harmonic analysis}

Let $G$ be a locally compact group and we denote the convolution algebra and the measure algebra of $G$ with $L^1(G)$ and $M(G)$, respectively. It is well-known that $L^1(G)$ is a 2-sided closed ideal in $M(G)$, so that for $\mu \in M(G)$ we have the following (left) multiplier
	$$M_\mu : L^1(G) \to L^1(G),\;f \mapsto \mu*f.$$
Moreover, it is also known that the multiplier algebra $M(L^1(G))$ of $L^1(G)$ can be identified with $M(G)$.

Let $P(G)$ be the set of all continuous positive definite functions on $G$ and let $B(G)$ be its linear span. The space $B(G)$ can be identified with the dual of the full group $C^*$-algebra $C^*(G)$, which is the completion of $L^1(G)$ under its largest $C^*$-norm. The space $B(G)$ with the pointwise multiplication and the dual norm is a commutative Banach algebra. The Fourier algebra $A(G)$ is the closure of $B(G)\cap C_c(G)$ in $B(G)$. It was shown in \cite{Em} that $A(G)$ is a commutative Banach algebra which is a 2-sided closed ideal in $B(G)$. Thus, any element $w\in B(G)$ gives rise to a multiplier on $A(G)$, but we actually have more multipliers. Recall that a function $w$ on $G$ is called a {\it multiplier} on $A(G)$ if $w\cdot A(G) \subseteq A(G)$. Then we have the following multiplier
	$$M_w: A(G) \to A(G),\; g\mapsto w\cdot g,$$
which is automatically bounded. We denote the collection of all multipliers on $A(G)$ by $MA(G)$. We define the space of all cb-multipliers $M_{cb}A(G)$ by
	$$M_{cb}A(G) := \{w\in MA(G) : ||M_w||_{cb}<\infty\}.$$
Both of the spaces are clearly commutative Banach algebras with respect to pointwise multiplication and we have the following inclusions
	$$A(G) \subseteq B(G) \subseteq M_{cb}A(G) \subseteq MA(G).$$

Let $G$ be a compact group. Then any element $\mu \in M(G)$ can be understood through its Fourier coefficients $(\widehat{\mu}(\pi))_{\pi \in \widehat{G}}$, where $\widehat{G}$ is the equivalence class of irreducible unitary representations on $G$ and 
	$$\widehat{\mu}(\pi) : = \int_G \pi(g^{-1})^t d\mu(g).$$
We use the notation
	$$\mu \sim \sum_{\pi\in \widehat{G}} d_\pi \text{Tr}(\widehat{\mu}(\pi) \pi^t),$$
which comes from the Fourier inversion formula saying that
	$$f(x) = \sum_{\pi\in \widehat{G}} d_\pi \text{Tr}(\widehat{f}(\pi) \pi^t(x))$$
for $f\in A(G) \cap L^1(G)$.

When $G$ is a discrete group we use a similar notation. Recall the translation operators $\lambda_g$, $g\in G$ on $\ell^2(G)$ given by $\lambda(g)(\delta_x) = \delta_{gx}$, $x\in G$, where $\delta_x$ is the point mass function on $x$. Now we consider the group von Neumann algebra $VN(G)$, the von Neumann algebra generated by $\{\lambda_g: g\in G\}$ in $B(\ell^2(G))$. Let  $\tau(\cdot)$ be the vacuum state on $VN(G)$ given by $\tau(\cdot) = \la \;\cdot \; \delta_e, \delta_e \ra$ and $L^2(VN(G))$ the associated $L^2$-space which is the completion of $VN(G)$ with respect to the inner product $\la \lambda_{g_1}, \lambda_{g_2} \ra := \tau(\lambda^*_{g_2}\lambda_{g_1})$. Note that we have $\ell^2(G) \cong L^2(VN(G))$ via the identification $\delta_g \mapsto \lambda_g$. Then clearly we have $VN(G) \subseteq L^2(VN(G))$ and any $T\in VN(G)$ is associated to a uniquely determined sequence $(\alpha_g)_{g\in G} \in \ell^2(G)$. In this case we write
	$$T \sim \sum_{g\in G}\alpha_g \lambda_g.$$
We will use the same notation for elements in $M_n(M(G))$, $M_n(VN(G))$ and $M_n(VN(G \times G))$.

\section{Construction of deformed convolution algebras and Fourier algebras}\label{sec-construction}

\begin{defn}
	Let $\mu\in M(G)$ and $w\in M_{cb}A(G)$ are norm 1 elements such that the corresponding multipliers $M_\mu$ and $M_w$ are injective with dense range. We define {\bf the deformed spaces} $L^1_\mu(G)$ and $A_w(G)$ as follows.
		$$L^1_\mu(G) := \mu*L^1(G),\;\; A_w(G):= w\cdot A(G)$$
	with the norms $||\cdot||_\mu$ and $||\cdot||_w$ given by
		$$||\mu*f||_\mu := ||f||_{L^1(G)} \;\; \text{and}\;\; ||w\cdot g||_w:= ||g||_{A(G)},\;\; f\in L^1(G),\; g\in A(G).$$
	In this case we call $\mu$ and $w$ by the {\bf deformation measure} for $L^1_\mu(G)$ and the {\bf deformation function} for $A_w(G)$, respectively. 
\end{defn}

\begin{rem}
	\begin{enumerate}
		\item
		The injectivity of the associated multipliers implies that the above norm formulas are well defined.
		\item In this case we have natural onto isometries:
			$$\Phi: L^1(G) \to L^1_\mu(G),\;\; f\mapsto \mu*f \;\;\text{and}\;\;  \Psi: A(G)\to A_w(G),\;\; g\mapsto w\cdot g.$$
		\item The natural operator space structures on $L^1_\mu(G)$ and $A_w(G)$ are the ones making $\Phi$ and $\Psi$ complete isometries, respectively.	
	\end{enumerate}
\end{rem}

\begin{prop}
The following maps are completely contractive.
	$$m_\mu: L^1_\mu(G) \prt L^1_\mu(G) \to L^1_\mu(G),\; f\otimes g \mapsto f*g$$
	and
	$$m_w: A_w(G) \prt A_w(G) \to A_w(G),\; f\otimes g \mapsto f\cdot g.$$
\end{prop}
\begin{proof}
By applying the complete isometry $\Phi$ we get
	$$\widetilde{m_\mu} = \Phi^{-1}\circ m_\mu \circ (\Phi\otimes \Phi): L^1(G)\prt L^1(G) \to L^1(G),\;\; f*\mu*g,$$
which is clearly a complete contraction, so that we know $m_\mu$ is also a complete contraction. The explanation for $m_w$ is the same.
\end{proof}

\begin{defn}
	The completely contractive Banach algebra $(L^1_\mu(G), m_\mu)$ is called {\it a deformed convolution algebra on $G$}. The completely contractive Banach algebra $(A_w(G), m_w)$ is called {\it a deformed Fourier algebra on $G$}.
\end{defn}

We record here the deformed mutiplication maps:
	$$\widetilde{m_\mu} = \Phi^{-1}\circ m_\mu \circ (\Phi\otimes \Phi): L^1(G)\prt L^1(G) \to L^1(G),\;\; f*\mu*g =: f*_\mu g,$$
	$$\widetilde{m_w} = \Psi^{-1}\circ m_w \circ (\Psi\otimes \Psi): A(G)\prt A(G) \to A(G),\;\; f\cdot w \cdot g =: f\cdot_w g.$$
\begin{rem}\label{rem-equiv-alg}
We can identify the algebras as completely contractive Banach algebras:
	$$(L^1_\mu(G), m_\mu) \cong (L^1(G), \widetilde{m_\mu}),\;\; (A_w(G), m_w) \cong (A(G), \widetilde{m_w}).$$
\end{rem}

We end this section by describing the duality for the spaces $L^1_\mu(G)$ and $A_w(G)$. We only consider $L^1_\mu(G)$ case since the other case is similar. Our understanding of the space $L^1_\mu(G)$ is based on the natural inclusion
	$$L^1_\mu(G) \subseteq L^1(G).$$
Recall that $L^1_\mu(G)$ is dense in $L^1(G)$, which is equivalent to the fact that the following restriction map is injective.
$$\iota: L^\infty(G) \cong (L^1(G))^* \hookrightarrow (L^1_\mu(G))^*,\;\; \varphi \mapsto \varphi|_{L^1_\mu(G)}.$$
This is why we require $M_\mu$ to have dense range, which is not so critical to the main results in this paper. Note that $L^\infty(G)$ (or more precisely $\iota(L^\infty(G))$) is clearly a $w^*$-dense subspace in $(L^1_\mu(G))^*$.
 
\begin{rem}
	\begin{enumerate}
		\item 
		Note that $L^1_\mu(G)$ is, in general, not closed in $L^1(G)$ with repect to the $L^1$-norm. 
		\item
		Recall the onto isometry $\Phi: L^1(G) \to L^1_\mu(G)$ gives us again an onto isometry
	$$\Phi^*: (L^1_\mu(G))^*\to L^\infty(G).$$ 
Now we can readily check that
	$$(\Phi^* \circ \iota) (g) = \check{\mu}*g,\;\; g\in L^\infty(G),$$
where $\check{\mu}$ is the measure given by
	$$\check{\mu}(E) = \mu(E^{-1}).$$	
	\end{enumerate}
\end{rem}

\section{Representability of the deformed convolution algebras on compact groups as operator algebras}\label{sec-compact-group}

In this section $G$ is always a compact group.

\subsection{The general case}

\begin{lem}
Let $\mu\in M(G)$ with norm 1. Then the associated multiplier $M_\mu$ is injective if and only if $\widehat{\mu}(\pi)$ is invertible for any $\pi\in \widehat{G}$. Moreover, in this case $M_\mu$ has dense range automatically.
\end{lem}
\begin{proof}
Note that $L_\mu$ is also decomposed into a direct sum of operators acting on finite dimensional spaces. Then this is trivial.
\end{proof}

In this section we will provide a chracterization of representatibility of the deformed algebra $L^1_\mu(G)$ as an operator algebra. For that purpose we focus on the equivalent algebra $(L^1(G), \widetilde{m_\mu})$ with the dual perspective. Indeed, we can easily see that $L^1_\mu(G)$ is completely representable as an operator algebra if and only if the map
	$$(\widetilde{m_\mu})^*: L^\infty(G) \to L^\infty(G) \bar{\otimes} L^\infty(G)$$	
extends to a completely bounded map $(\widetilde{m_\mu})^*: L^\infty(G) \to L^\infty(G) \otimes_{eh} L^\infty(G)$. Here $\bar{\otimes}$ is the spatial tensor product of von Neumann algebras. We call the map $(\widetilde{m_\mu})^*$, {\it the deformed co-multiplication}. We first need to know how the deformed co-multiplication $(\widetilde{m_\mu})^*$ acts on concrete elements of $L^\infty(G)$.

\begin{prop}
For $f\in L^\infty(G)$ we have
	$$(\widetilde{m_\mu})^*(f)(x,y) = \int_G f(xzy)d\mu(z),\;\; x,y\in G.$$
In particular for $f=\pi_{ij}$ we have
	\begin{equation}\label{eq-deformed-comulti}(\widetilde{m_\mu})^*(\pi_{ij}) = \sum^{d_\pi}_{k=1}\pi_{ik} \otimes [\widehat{\mu}(\bar{\pi})\pi]_{kj},
	\end{equation}
where $\bar{\pi}$ is the conjugate representation of $\pi$ and $\overline{A}$ is the matrix complex conjugate given by $[\overline{A}]_{ij} = \overline{a_{ij}}$.	
\end{prop}
\begin{proof}
For $g,h \in L^1(G)$ we have
	\begin{align*}
	\la (\widetilde{m_\mu})^*(f), g\otimes h \ra
	&= \la f, g*\mu*h \ra\\
	& = \int_G\int_G \int_G g(x)h(z^{-1}x^{-1}y)f(y)d\mu(z)dxdy\\
	& = \int_G\int_G \int_G g(x)h(y)f(xzy)d\mu(z)dxdy\\
	& = \int_G\int_G \left(\int_G f(xzy)d\mu(z) \right)g(x)h(y)dxdy.
	\end{align*}

In particular, for $f = \pi_{ij}$ we have
	\begin{align*}
	(\widetilde{m_\mu})^*(\pi_{ij}) & = \int_G \pi_{ij}(xzy)d\mu(z)\\
	& = \left[\int_G \pi(xzy)d\mu(z)\right]_{ij}\\
	& = \left[\pi(x)\left(\int_G \pi(z)d\mu(z)\right)\pi(y)\right]_{ij},
	\end{align*}
which gives us the desired conclusion.
\end{proof}

We also need the following theorem by S. Helgason, which is a compact group generalization of a Littlewood's theorem.
\begin{thm}\label{thm-Helgason}({\bf Helgason \cite{Hel}})\\
Let $A = (A^\pi)_{\pi\in \widehat{G}}$ be any family of matrices with $A^\pi \in M_{d_\pi}$. Suppose that the formal series
	$$\sum_{\pi\in \widehat{G}} d_\pi \text{Tr}(A^\pi U^\pi \pi^t)$$
belong to $M(G)$ for any choice of unitary $U = (U^\pi)_{\pi\in \widehat{G}}$, $U^\pi \in \mathcal{U}(d_\pi)$.	Then we have
	$$\sum_{\pi\in \widehat{G}} d_\pi \text{Tr}((A^\pi)^*A^\pi)<\infty.$$
\end{thm}

Here comes the main results of this section.

\begin{thm}\label{thm-main-L1}
The deformed algebra $L^1_\mu(G)$ is completely representable as an operator algebra if $\mu\in L^2(G)$. The converse is true when $\mu$ is central, i.e. $\widehat{\mu}(\pi) = c_\pi I_\pi$ for some $c_\pi$ for all $\pi\in \widehat{G}$. Moreover, the same holds in the category of Banach spaces.
\end{thm}
\begin{proof}
First we suppose that $\mu\in L^2(G)$. If we pick $F = [f_{kl}]_{kl}\in M_n(L^\infty(G))$, then we have
	$$F \sim \sum_{\pi\in \widehat{G}} d_\pi \sum^{d_\pi}_{i,j=1}\widehat{F}(\pi)_{ij} \otimes \pi_{ij},$$
where $\widehat{F}(\pi)_{ij} = \left[(\widehat{f_{kl}})_{ij}\right]_{kl} \in M_n$. Then by \eqref{eq-deformed-comulti} we have
	\begin{align*}
	(I_n\otimes (\widetilde{m_\mu})^*)(F) & \sim \sum_{\pi\in \widehat{G}} d_\pi \sum^{d_\pi}_{i,j,k=1}\widehat{F}(\pi)_{ij} \otimes \pi_{ik} \otimes [\widehat{\mu}(\bar{\pi})\pi]_{kj}\\
	& = \sum_{\pi\in \widehat{G}}\sum^{d_\pi}_{j,k=1} \left(\sum^{d_\pi}_{i=1}\sqrt{d_\pi} \widehat{F}(\pi)_{ij} \otimes \pi_{ik}\right) \odot \left(\sqrt{d_\pi} I_n \otimes [\widehat{\mu}(\bar{\pi})\pi]_{kj}\right)\\
	& = A\odot B,
	\end{align*}
where $A$ and $B$ are row and column matrices, respectively, given as follows. For the index set $J = \{(\pi, j,k) : \pi \in \widehat{G}, 1\le j,k\le d_\pi\}$, the matrix $A$ is the $M_n(L^\infty(G))$-valued row matrix in $M_{J, J}$ whose $(\pi, j,k)$-th entry is $\sum^{d_\pi}_{i=1}\sqrt{d_\pi} \widehat{F}(\pi)_{ij} \otimes \pi_{ik}\in M_n(L^\infty(G))$. Similarly, the matrix $B$ is the $M_n(L^\infty(G))$-valued column matrix in $M_{J, J}$ whose $(\pi, j,k)$-th entry is $\sqrt{d_\pi} I_n \otimes [\widehat{\mu}(\bar{\pi})\pi]_{kj}\in M_n(L^\infty(G))$. Moreover, we have
	\begin{align*}
	||A||^2_{M_{J,J}(M_n(L^\infty(G)))}
	& = ||\sum_{\pi\in \widehat{G}}\sum^{d_\pi}_{j,k=1} \left(\sum^{d_\pi}_{i,i'=1}d_\pi \widehat{F}(\pi)_{ij} \widehat{F}(\pi)_{i'j}^*\otimes \pi_{ik}\overline{\pi_{i'k}}\right)||_{M_n(L^\infty(G))}\\
	& = ||\sum_{\pi\in \widehat{G}}\sum^{d_\pi}_{j,i,i'=1} d_\pi \widehat{F}(\pi)_{ij} \widehat{F}(\pi)_{i'j}^*\delta_{i,i'}\otimes 1_G||_{M_n(L^\infty(G))}\\
	& = ||\sum_{\pi\in \widehat{G}}\sum^{d_\pi}_{i,j=1} d_\pi \widehat{F}(\pi)_{ij} \widehat{F}(\pi)_{ij}^*||_{M_n}\\
	& = ||\int_G F(x)F^*(x) dx||_{M_n} \le ||F||^2_{M_n(L^\infty(G))}.
	\end{align*}

For $B$ we have
	\begin{align*}
	||B||^2_{M_{J,J}(M_n(L^\infty(G)))}
	& = ||\sum_{\pi\in \widehat{G}}\sum^{d_\pi}_{j,k=1} d_\pi (I_n)^*I_n \otimes \overline{[\widehat{\mu}(\bar{\pi})\pi]_{kj}}[\widehat{\mu}(\bar{\pi})\pi]_{kj}||_{M_n(L^\infty(G))}\\
	& = ||I_n \otimes \sum_{\pi\in \widehat{G}}d_\pi \sum^{d_\pi}_{j,k=1} |[\widehat{\mu}(\bar{\pi})]_{kj}|^2\cdot 1_G||_{M_n(L^\infty(G))}\\
	& = \sum_{\pi\in \widehat{G}}d_\pi ||\widehat{\mu}(\bar{\pi})||^2_{S^2_{d_\pi}} = ||\mu||^2_{L^2(G)}.
	\end{align*}
Note that we use unitarity of $\pi$ for the second equality. Now combining the above two we get
	$$||(I_n\otimes (\widetilde{m_\mu})^*)(F)||_{eh} \le ||F||_{M_n(L^\infty(G))}\cdot ||\mu||_{L^2(G)},$$
which implies that $(\widetilde{m_\mu})^*$ is completely bounded with cb-norm $\le ||\mu||_{L^2(G)}$.

For the converse direction we assume that $L^1_\mu(G)$ is completely representable as an operator algebra and $\mu$ is central with $\widehat{\mu}(\pi) = c_\pi I_\pi$, $\pi\in \widehat{G}$. If we take any $f\in L^\infty(G)$, then we have
	$$f \sim \sum_{\pi\in \widehat{G}} d_\pi \sum^{d_\pi}_{i,j=1}\widehat{f}(\pi)_{ij} \pi_{ij}.$$
By \eqref{eq-deformed-comulti} we also have
	$$(\widetilde{m_\mu})^*(f) \sim \sum_{\pi\in \widehat{G}} c_{\bar{\pi}} d_\pi \sum^{d_\pi}_{i,j,k=1}\widehat{f}(\pi)_{ij}\pi_{ik} \otimes \pi_{kj}.$$
Now we will use the embedding $L^\infty(G) \otimes_{eh} L^\infty(G) \hookrightarrow CB^\sigma(B(L^2(G))).$	For any choice of unitary $U = (U^\pi)_{\pi\in \widehat{G}}$, $U^\pi \in \mathcal{U}(d_\pi)$, we define $T_U \in B(L^2(G))$ by
	$$T_U(\pi_{ij}) := [\pi^*U^\pi]_{ij}.$$
We can actually check $T_U$ is an isometry. Indeed, for any $\pi \in \widehat{G}$ we have
	\begin{align*}
		\la T_U(\pi_{ij}), T_U(\sigma_{kl}) \ra
		& = \sum^{d_\pi}_{r,s=1} \la \bar{\pi}_{ri}U^\pi_{rj}, \bar{\sigma}_{sk}U^\pi_{sl}\ra\\
		& = \sum^{d_\pi}_{r,s=1} \frac{1}{d_\pi}\delta_{\pi,\sigma}\delta_{rs}\delta_{ik}U^\pi_{rj}\overline{U^\pi_{sl}}\\
		& = \frac{\delta_{\pi,\sigma}\delta_{rs}\delta_{ik}}{d_\pi} = \la \pi_{ij}, \sigma_{kl}\ra.
	\end{align*}
Now we recall the embedding $L^\infty(G) \hookrightarrow B(L^2(G)),\; f\mapsto M_f$, where $M_f$ is the multiplication operator with respect to $f$, so that we have
	\begin{align*}
	(\widetilde{m_\mu})^*(f)(T_U)(1_G)
	& = \sum_{\pi\in \widehat{G}} c_{\bar{\pi}} d_\pi \sum^{d_\pi}_{i,j,k=1}\widehat{f}(\pi)_{ij}M_{\pi_{ik}}\circ T_U \circ M_{\pi_{kj}}(1_G)\\
	& = \sum_{\pi\in \widehat{G}} c_{\bar{\pi}} d_\pi \sum^{d_\pi}_{i,j,k,l=1}\widehat{f}(\pi)_{ij}\pi_{ik}\bar{\pi}_{lk}U^\pi_{lj}\\
	& = \sum_{\pi\in \widehat{G}} c_{\bar{\pi}} d_\pi \sum^{d_\pi}_{i,j=1}\widehat{f}(\pi)_{ij}[\pi \pi^* U^\pi]_{ij}\\
	& = \sum_{\pi\in \widehat{G}} c_{\bar{\pi}} d_\pi \text{Tr}(\widehat{f}(\pi)[U^\pi]^t) \cdot 1_G\\
	&  = \la (\widehat{f}(\pi))_{\pi\in \widehat{G}},  (c_{\bar{\pi}} [U^\pi]^t)_{\pi\in \widehat{G}}\ra \cdot 1_G
	\end{align*}
for $f\in \text{Pol}(G) := \text{span}\{\pi_{ij} : \pi\in \widehat{G}, 1\le i,j\le d_\pi\}$.	
Since $\text{Pol}(G)$ is dense in $C(G)$ we have
	\begin{align*}
	||(\widetilde{m_\mu})^*|| & \ge \sup\{ ||(\widetilde{m_\mu})^*(f)(T_U)(1_G)||_{L^2(G)}: ||f||_{C(G)} \le 1,\, f\in \text{Pol}(G)\}\\
	& = \sup\{ |\la (\widehat{f}(\pi))_{\pi\in \widehat{G}},  (c_{\bar{\pi}} [U^\pi]^t)_{\pi\in \widehat{G}}\ra |: ||f||_{C(G)} \le 1,\, f\in \text{Pol}(G)\}\\
	& = || \sum_{\pi\in \widehat{G}} c_{\bar{\pi}} d_\pi \text{Tr}([U^\pi]^t\pi^t) ||_{M(G)}.
	\end{align*}
Now we appeal to Theorem \ref{thm-Helgason} to get the conclusion we wanted.

The result in the Banach space category follows easily from the above calculation and \cite[Theorem 2.8]{LSS}. Note that \cite[Theorem 2.8]{LSS} deals with discrete groups but the same proof works for general locally compact groups.
\end{proof}

\subsection{The case of compact connected Lie groups and their real dimensions}\label{subsec-compact-connected}
In this subsection we apply Theorem \ref{thm-main-L1} in the case of compact connected Lie groups with the deformation measures coming from the Laplacian on the group. We will demonstrate that the representability of $L^1_\mu(G)$ can precisely detect the dimension of the group $G$.

\begin{ex}\label{ex-probability}
Let $\Omega_\pi = -\kappa_\pi I_\pi$ is the Casimir operator (in other words, Laplacian on $G$) for $\pi\in \widehat{G}$ with $\kappa_\pi \ge 0$. Then, there is a family of probability measures $\mu_t$, $t>0$ on $G$ such that
	$$\widehat{\mu_t}(\pi) = e^{-t\kappa_\pi}I_\pi.$$ 
Using $\mu_t$ we could find probability measures with polynomially decreasing Fourier coefficients by a standard argument. For $\alpha>0$ we recall the formula
	$$(1+n)^{-\alpha} = \frac{1}{\Gamma(\alpha)}\int^\infty_0 t^{\alpha-1}e^{-t}e^{-tn}dt,$$
where $\Gamma(\alpha) = \int^\infty_0 t^{\alpha-1}e^{-t}dt$ is the Gamma function. We define
	$$\nu_\alpha := \frac{1}{\Gamma(\alpha)}\int^\infty_0 t^{\alpha/2-1}e^{-t}\mu_t\,dt$$
Then, $\nu_\alpha$	is clearly a probability measure with
	\begin{equation}\label{eq-poly-measure}
	\widehat{\nu_\alpha}(\pi) = \frac{1}{(1+\kappa_\pi)^{\alpha/2}}I_\pi,\;\; \pi\in \widehat{G}.
	\end{equation}
\end{ex}

Now we recall some standard Lie theory we need. See \cite{Wall} or \cite[section 5]{LST} for the details. Let $\mathfrak{g}$ be the Lie algebra of $G$ with the decomposition $\mathfrak{g} = \mathfrak{z} + \mathfrak{g}_1$,
where $\mathfrak{z}$ is the center of $\mathfrak{g}$ and $\mathfrak{g}_1 = [\mathfrak{g}, \mathfrak{g}]$.
Let $\mathfrak{t}$ be a maximal abelian subalgebra of $\mathfrak{g}_1$ and $T = \text{exp}\mathfrak{t}$.
Then there are fundamental weights $\lambda_1, \cdots, \lambda_r, \Lambda_1,\cdots,\Lambda_l \in \mathfrak{g}^*$ with $r = \text{dim}\mathfrak{z}$ and $l=\text{dim}\mathfrak{t}$ such that
any $\pi \in \widehat{G}$ is in one-to-one correspondence with its associated highest weight $\Lambda_\pi = \sum^r_{i=1}a_i\lambda_i + \sum^l_{j=1}b_j \Lambda_j$ with
$(a_i)^r_{i=1}\in \z^r$ and $(b_j)^l_{j=1}\in \z^l_+$. The 1-norm $||\pi||_1$ of $\pi$ is given by
	$$||\pi||_1 := \sum^r_{i=1}|a_i| + \sum^l_{j=1} b_j.$$
This 1-norm is known to be equivalent to $\sqrt{\kappa_\pi}$ from the Casimir operator. More precisely, there are positive constants $c_1$ and $c_2$ independent of $\pi$ such that (\cite[Lemma 5.6.6]{Wall})
	$$c_1||\pi||^2_1 \le  \kappa_\pi \le c_2||\pi||^2_1.$$
Moreover, the following summability condition is known.

\begin{prop}\label{prop-DR}(\cite[Lemma 3.1]{DR})
For $\alpha>0$ we have $\displaystyle \sum_{\pi\in \widehat{G}}\frac{d^2_\pi}{(1+||\pi||_1)^{2\alpha}}<\infty$ if and only if $\alpha > \frac{d(G)}{2}$, where $d(G)$ is the dimension of $G$ as a real Lie group.
\end{prop}

By combining Theorem \ref{thm-main-L1} and Proposition \ref{prop-DR} we get the following.

\begin{thm}
Let $\nu_\alpha$ is the probability measure from \eqref{eq-poly-measure}. Then, $L^1_{\nu_\alpha}(G)$ is (completely) representable as an operator algebra if and only if $\alpha > \frac{d(G)}{2}$.
\end{thm}

\begin{rem}\label{rem-main-L1}
	\begin{enumerate}
		\item
		The above theorem tells us that the representability of  $L^1_{\nu_\alpha}(G)$ as an operator algebra precisely detects the dimension of the underlying group.
		\item
		The 1-norm is also known to be equivalent to the length function $\tau$ on $\widehat{G}$, which is given as follows. Let $\chi_i$ be the character of $G$ associated to the highest weight $\lambda_i$ and $\pi_j$ be the irreducible representation associated to the weight $\Lambda_j$. It is well-known that $S=\{\pm \chi_i, \pi_j : 1\le i\le r, 1\le j \le l\}$ generates $\widehat{G}$, i.e. $\bigcup_{k\ge 1} S^{\otimes k} = \widehat{G}$, where $S^{\otimes k} =\{\pi \in \widehat{G} : \pi \subset \sigma_1 \otimes \cdots \otimes \sigma_k \;\; \text{where}\;\; \sigma_1, \cdots, \sigma_k \in S\cup\{1\} \}$, $k\ge 1$. Then, $\tau  : \widehat{G} \to \mathbb{N}\cup\{0\}$ is given by
	$$\tau(\pi) := k,\;\;\text{if}\;\; \pi\in S^{\otimes k}\backslash S^{\otimes(k-1)}.$$
Then, we have some constant $C>0$ such that (see the proof of \cite[Theorem 5.4]{LST})
	$$||\pi||_1 \le C \tau(\pi) \le C||\pi||_1,\; \pi \in \widehat{G}.$$	
This justifies the statement that the quantity $\sqrt{\kappa_\pi}$ describes the growth rate of $\widehat{G}$.
		\item
		In the case of compact connected Lie groups we have a replacement of Theorem \ref{thm-Helgason} in the proof of the negative direction of Theorem \ref{thm-main-L1} as follows. Suppose that  $L^1_\mu(G)$ is completely representable as an operator algebra and $\mu$ is central with $\widehat{\mu}(\pi) = c_\pi I_\pi$, $\pi\in \widehat{G}$. Recall that For any choice of unitary $U = (U^\pi)_{\pi\in \widehat{G}}$, $U^\pi \in \mathcal{U}(d_\pi)$, we have
	$$||(\widetilde{m_\mu})^*|| \ge || \sum_{\pi\in \widehat{G}} c_\pi d_\pi \text{Tr}([U^\pi]^t\pi^t) ||_{M(G)}.$$
Now we put $U^\pi = r_\pi I_\pi$, where $(r_\pi)_{\pi\in \widehat{G}}$ is an I.I.D. family of Bernoulli variables. We use cotype 2 condition of $M(G)$ to get
	$$\mathbb{E} || \sum_{\pi\in \widehat{G}} c_\pi d_\pi r_\pi \text{Tr}(\pi^t) ||_{M(G)}\ge C (\sum_{\pi\in \widehat{G}} d^2_\pi |c_\pi|^2||\chi_\pi||^2_1)^{\frac{1}{2}}$$
for some constant $C>0$, where $\chi_\pi = \text{Tr}(\pi) = \text{Tr}(\pi^t)$ is the character function associated to $\pi$. Recall that a Banach space $X$ is said to be of cotype 2 if there is a constant $D>0$ such that
	$$\mathbb{E} || \sum_i r_i x_i ||_{X}\ge D (\sum_i ||x_i||^2_X)^{\frac{1}{2}}$$
for any $(x_i) \subseteq X$ and an I.I.D. family of Bernoulli variables $(r_i)$. Note that the dual of a $C^*$-algebra is known to be of cotype 2 (\cite{Fa}) and $M(G) \cong C_0(G)^*$. We finally note that that there is a constant $C'>0$ such that
	$$||\chi_\pi||_1 \ge C'$$
for any $\pi \in \widehat{G}$(\cite{Pr}), which leads us to the conlusion we wanted.

\end{enumerate}
\end{rem}

\section{Representability of the deformed Fourier algebras on discrete groups as operator algebras}

In this section $G$ is always a discrete group.

\begin{lem}
Let $w \in M_{cb}A(G)$ with norm 1. Then the associated multiplier $L_w$ is injective with dense range if and only if $w$ has no zero value.
\end{lem}
\begin{proof}
Suppose that $w$ has no zero value, then $L_w$ has dense range since the image of $L_w$ contains $\{\delta_x: x\in G\}$, which is dense in $A(G)$. The injectivity of $L_w$ is clear. Conversely, if  $L_w$ is injective (with dense range), then $w(x)\delta_x\neq 0$ for all $x\in G$, so that $w(x)\neq 0$ for all $x\in G$.
\end{proof}

In this section we will provide a characterization of complete representatibility of the deformed algebra $A_w(G)$ as an operator algebra. As before we focus on the equivalent algebra $(A(G), \widetilde{m_w})$ with the dual perspective. Indeed, we can easily see that $A_w(G)$ is completely representable as an operator algebra if and only if the map
	$$(\widetilde{m_w})^*: VN(G) \to VN(G) \bar{\otimes} VN(G)$$
extends to a completely bounded map $(\widetilde{m_w})^*: VN(G) \to VN(G) \otimes_{eh} VN(G)$. We call the map $(\widetilde{m_w})^*$, {\it the deformed co-multiplication} as before. This time it is quite straight forward to see how $(\widetilde{m_w})^*$ acts on concrete elements of $VN(G)$, so we just record it without proof.

\begin{prop}
For $T \sim \sum_{g\in G}\alpha_g \lambda_g \in VN(G)$ we have
	$$(\widetilde{m_w})^*(T) \sim \sum_{g\in G}w(g)\alpha_g \lambda_g \otimes \lambda_g.$$
\end{prop}

For the negative direction of our main result we need the following Lemma, which is a direct consequence of Lust-Piquard's non-commutative version of Kahane-Katznelson-de Leeuw's coefficient problem.

\begin{lem}\label{lem-coefficient-VN(G)}
Let $\tau(\cdot) = \la \;\cdot \; \delta_e, \delta_e \ra$ be the vacuum state on $VN(G)$. There is a constant $K>0$ such that for any $(c_g) \in \ell^2(G)$ with norm $\le 1$ there exist $T\in VN(G)$ with norm $\le K$ such that
	$$|\tau(T\lambda^*_g)|\ge |c_g|,\; g\in G.$$
\end{lem}
\begin{proof}
Fix a sequence $(c_g)_{g\in G}$ with $\sum_g |c_g|^2\leq 1$ and set $A_g =c_g\lambda_g$. Here, we may assume that the index set of sequence is countable. Then, the sequence $(A_g)$ satisfies the conditions (i) and (ii) in \cite[Theorem 4]{Piq}. More precisely, for $g_1,g_2,g_3(g_2\neq g_3)$ we have
	$$\text{Re}(\tau(A_{g_1}^*A_{g_1}A_{g_2}^*A_{g_3}))=|c_{g_1}|^2 \text{Re}(\tau(\overline{c_{g_2}}c_{g_3}\lambda_{g_2^{-1}g_3}))=0$$
and
	$$\text{Re}(\tau(A_{g_1}A_{g_1}^*A_{g_2}A_{g_3}^*))=|c_{g_1}|^2 \text{Re}(\tau(c_{g_2}\overline{c_{g_3}}\lambda_{g_2g_3^{-1}}))=0.$$
Moreover,  we have
	$$||\sum_g|A_g|_s^2|| = ||\sum_g \frac{A_g^*A_g+A_gA_g^*}{2}|| = \sum_g |c_g|^2\le 1.$$ By \cite[Theorem 4]{Piq} there is $T\in VN(G)$ with norm $\le K$ such that
	$$|c_g||\tau(T^*\lambda_g)| = |\tau(T^*A_g)| \ge \tau(A_g^*A_g)=|c_g|^2,\; g\in G,$$
which is the conclusion we wanted.	
\end{proof}

Here comes the main result of this section.

\begin{thm}\label{thm-main-A(G)}
The deformed algebra $A_w(G)$ is completely representable as an operator algebra if and only if $w\in \ell^2(G)$.
\end{thm}
\begin{proof}
We first check the positive direction. Suppose that
	$$T \sim \sum_{g\in G}A_g \lambda_g \in M_n(VN(G))$$
with $A_g\in M_n$. Then, we have
	\begin{align*}
	(I_n\otimes m_w)^*(T) & \sim \sum_{g\in G}w(g)A_g \otimes \lambda_g \otimes \lambda_g\\
	& = \sum_{g\in G} (A_g \otimes \lambda_g) \odot (w(g)id_n \otimes \lambda_g)\\
	& = \begin{bmatrix} \; \cdots & A_g\otimes \lambda_g & \cdots \;\; \end{bmatrix} \odot \begin{bmatrix} \vdots \\ w(g)id_n \otimes \lambda_g \\ \vdots \end{bmatrix}\\
	& = A \odot B.
	\end{align*}
Now we have $A\in M_G(M_n(VN(G)))$ and
	$$||A||^2 = || \sum_{g\in G} (A_g \otimes \lambda_g)(A_g \otimes \lambda_g)^* || = || \sum_{g\in G} A_gA^*_g||_{M_n}.$$
Moreover we have $B\in M_G(M_n(VN(G)))$ and
	$$||B||^2 = || \sum_{g\in G} (w(g)id_n \otimes \lambda_g)^*(w(g)id_n \otimes \lambda_g) || = \sum_{g\in G} |w(g)|^2.$$
Finally we observe that
	$$|| \sum_{g\in G} A_gA^*_g||_{M_n} \le ||TT^*||.$$
Indeed, let $\tau(\cdot) = \la \;\cdot \; \delta_e, \delta_e \ra$ be the vacuum state. Then we have
	$$(id_n \otimes \tau)(TT^*) = \sum_{g\in G} A_gA^*_g.$$
Combining all the above we get
	$$||\widetilde{m_w}||_{cb} \le ||w||_{\ell^2(G)}.$$
For the converse direction we let $T \sim \sum_{g\in G}\alpha_g \lambda_g \in VN(G)$. Moreover, we set $X_r\in B(\ell^2(G))$ by
	$$X_r(\delta_g) = r_g \delta_{g^{-1}},\; g\in G$$
where $r_g\in \Comp$	with $|r_g|=1$, $g\in G$. Then, we have
	$$(\widetilde{m_w})^*(T)(X_r)\delta_e = \left(\sum_{g\in G} w(g)\alpha_g r_g \right)\delta_e.$$
This allows us the follows estimate.
	\begin{equation}\label{eq-estimate-VN(G)}
	||(\widetilde{m_w})^*|| \ge \frac{|\sum_{g\in G} w(g)\alpha_g r_g|}{||T||_{VN(G)}}.
	\end{equation}
Now we fix $(c_g) \in \ell^2(G)$ with norm $\le 1$. Then, there is $T\in VN(G)$ (depending on $(c_g)$) with norm $\le K$ such that $|\tau(T\lambda^*_g)|\ge |c_g|$, $g\in G$ by Lemma \ref{lem-coefficient-VN(G)}. Then we choose $r_g$ (depending on $T$) so that $w(g)\alpha_g r_g = |w(g)\alpha_g| = |w(g)\tau(T\lambda^*_g)|$, $g\in G$. For these choice of $T$ and $(r_g)$ we get
	$$||(\widetilde{m_w})^*|| \ge \frac{1}{K} \sum_{g\in G} |w(g)\alpha_g| \ge \frac{1}{K} \sum_{g\in G} |w(g)c_g|.$$
The choice of 	$(c_g)$ is arbitrary, so that we get
	$$||(\widetilde{m_w})^*|| \ge \frac{1}{K} ||w||_{\ell^2(G)}.$$
\end{proof}

\begin{rem}
We also have a different route for the negative direction of the above theorem with a little bit of probabilistic flavor as in Remark \ref{rem-main-L1}, which allows us to avoid Lemma \ref{lem-coefficient-VN(G)}. We consider an I.I.D. family of Bernoulli variables $(r_g)_{g\in G}$. Then \eqref{eq-estimate-VN(G)} tells us that
	\begin{align*}
	||(\widetilde{m_w})^*||
	& \ge \mathbb{E}||\sum_{g\in G} w(g)\lambda_g r_g||_{(C^*_r(G))^*}\\
	& \ge C \left(\sum_{g\in G} ||w(g)\lambda_g ||^2_{(C^*_r(G))^*}\right)^{\frac{1}{2}} = C\cdot ||w||_{\ell^2(G)}
	\end{align*}
for some constant $C>0$. Here we used the fact that $(C^*_r(G))^*$ is of cotype 2 and $\lambda_g$ is understood as an element of $(C^*_r(G))^*$ given by $a \mapsto \tau(\lambda_g a)$.
\end{rem}

\subsection{The case of finitely generated groups and the growth order}\label{subsec-f.g.}
In this subsection we apply Theorem \ref{thm-main-A(G)} in the case of certain finitely generated groups with the deformation functions coming from the word length function. We will demonstrate that the complete representability of $A_w(G)$ can precisely detect the growth rate of the group $G$.

\begin{ex}\label{ex-PD-function}
Let $G$ be a finitely generated group with a generating set $S$. Let $|\cdot|$ be the word length associated to $S$. We consider the function given by
	$$w_t(g) = e^{-t|g|},\; g\in G.$$
If $G$ is either $\z^n$, a Coxeter group with the canonical generating set $S$, then $w_t$ is known to be a positive definite function. If $G$ is a hyperbolic group with the canonical generating set, then it is known that $w_t \in M_{cb}A(G)$ with
	$$M := \sup_{t>0}||w_t||_{M_{cb}A(G)} <\infty.$$
Note that $M=1$ in the previous case. Using $w_t$ we could find polynomially decreasing functions in $M_{cb}A(G)$ as before. For $\alpha>0$ we consider the function
	\begin{equation}\label{eq-poly-PD-function}
	W_\alpha(g) := \frac{1}{M(1+|g|)^\alpha},\; g\in G.
	\end{equation}
since we have
	$$W_\alpha = \frac{1}{\Gamma(\alpha)}\int^\infty_0 t^{\alpha/2-1}e^{-t}w_t\,dt$$
we have
	$$||W_\alpha||_{M_{cb}A(G)} \le \frac{1}{M\Gamma(\alpha)}\int^\infty_0 t^{\alpha/2-1}e^{-t}||w_t||_{M_{cb}A(G)}\,dt \le 1.$$
\end{ex}

Recall that a finitely generated group with a fixed generating set $S$ is said to be {\it polynomially growing} if there is a constant $C>0$ and $k>0$ such that
	$$|B(n)|\le C(n^k+1),\; n\ge 0,$$
where $B(n) =\{g\in G:|g|\le n \}$ is the $n$-ball. Recall also that the infimum $k_0$ of such $k$ is called the order of the polynomial growth of $G$ or the growth rate of $G$. Moreover, $G$ is said to be {\it exponentially growing} if there is a constant $a>1$ such that
	$$|B(n)|\ge a^n,\; n\ge 0.$$
The exponential growh rate of $G$ with respect to $S$ is defined by 
	$$\lambda = \lambda(G,S):=\lim_{n\rightarrow \infty}|B(n)|^{\frac{1}{n}}.$$	
\begin{rem}
	\begin{enumerate}
		\item For the group $G$ of polynomial growth, it is well known that $k_0$ has to be a natural number and
			$$|B(n)|\sim n^{k_0}, \; n\ge 0,$$
		i.e. there are constants $C_1, C_2>0$ such that
			$$C_1 n^{k_0} \le |B(n)| \le C_2 n^{k_0}, \; n\ge 0.$$	
		\item
		The polynomial growth rate $k_0$ of $G$ is known to be independent of the choice of the generating set $S$.
		\item
		Every finitely generated group has at most exponential growth. In other words, there is $b>0$ such that $|B(n)|\leq b^n$ for all $n$. Thus the above limit $\lambda(G,S)$ always exists by Fekete's subadditivity lemma.
		\item
		The condition $\lambda(G,S)>1$ implies that $\lambda(G,S')>1$ for any other symmetric generating set $S'$. However, $\inf_S \lambda(G,S)$ could be equal to 1.
	\end{enumerate}
\end{rem}

\begin{prop}
Let $G$ be a finitely generated group with a fixed generating set $S$.
	\begin{enumerate}
		\item Suppose $G$ is of polynomial growth of order $k_0$. For $\alpha>0$ we have
			$$\displaystyle \sum_{g\in G}\frac{1}{(1+|g|)^{2\alpha}}<\infty$$
		if and only if $\alpha > \frac{k_0}{2}$.
		\item Suppose that $G$ is exponential growing with the growth rate $\lambda$. For $t>0$ we have $\displaystyle \sum_{g\in G}\frac{1}{e^{2t|g|}}<\infty$ if $t > \frac{\log \lambda}{2}$ and $\displaystyle \sum_{g\in G}\frac{1}{e^{2t|g|}} = \infty$ if $t < \frac{\log \lambda}{2}$.
	\end{enumerate}
\end{prop}
\begin{proof}
We only consider the case (1) since the proof for (2) is essentially the same. Put $C(n):=\left \{g\in G:|g|=n \right \}$, the $n$-sphere, $c_n:=|C(n)|$ and $b_n := |B(n)|$. Then we have
	$$\sum_{g\in G}\frac{1}{(1+|g|)^{2\alpha}}=\sum_{n\geq 0}\frac{c_n}{(1+n)^{2\alpha}}.$$
By summation by parts we have
	$$\sum_{n\geq 0}\frac{c_n}{(1+n)^{2\alpha}}=\frac{b_N}{(1+N)^{2\alpha}}+\sum_{n\geq 0}^{N-1}b_n\left(\frac{1}{(n+1)^{2\alpha}}-\frac{1}{(n+2)^{2\alpha}}\right).$$
Recall that
	$$b_n \sim n^{k_0},\; n\ge 0$$
and by the mean value theorem we have
	$$\frac{1}{(n+1)^{2\alpha}}-\frac{1}{(n+2)^{2\alpha}} \sim \frac{2\alpha}{(n+1)^{2\alpha+1}},\; n\ge 0.$$
Then a standard summability criterion gives us the conclusion we wanted.

\end{proof}

\begin{thm}
Let $G$ be one of $\z^n$, a Coxeter group or a hyperbolic group, and $w_t$ and $W_\alpha$ are the functions from Example \ref{ex-PD-function} with the canonical generating set $S$.
	\begin{enumerate}
		\item
		Suppose $G$ is of polynomial growth of order $k_0$. The algebra $A_{W_\alpha}(G)$ is completely representable as an operator algebra if and only if $\alpha > \frac{k_0}{2}$.
		\item
		Suppose that $G$ is exponential growing with the growth rate $\lambda$. The algebra $A_{w_t}(G)$ is completely representable as an operator algebra if  $t> \frac{\log \lambda}{2}$ and $A_{w_t}(G)$ is not completely representable as an operator algebra if  $t< \frac{\log \lambda}{2}$.
	\end{enumerate}	
\end{thm}

\begin{rem}
	\begin{enumerate}
		\item
		The above theorem tells us that the complete representability of $A_w(G)$ can precisely detect the growth rate of the underlying group when it is polynomially growing or exponentially growing. 
		\item
		We excluded the case of exponentially growing deformation measures in Section \ref{subsec-compact-connected} since the dual of compact connected Lie groups are always of polynomial growth (\cite{Ver}).
	\end{enumerate}
\end{rem}

{\bf Acknowledgement.}  The authors are grateful to Nico Spronk and Ebrahim Samei for their comments on defining deformed algebras.

\end{document}